\documentclass[12pt]{article}
\usepackage{amsfonts}
\usepackage{}
\usepackage{amsmath,amsfonts,amssymb,amsthm}
\newcommand{\1}{{{\bf 1}}}
\newcommand{\id}{{\rm id}}

\newcommand{\End}{{\rm End}\,}
\newcommand{\Res}{{\rm Res}\,}

\newcommand{\w}{{\omega}}

\newcommand{\Z}{\mathbb{Z}}

\newcommand{\C}{\mathbb{C}}

\newcommand{\wt}[1]{{\rm wt \hskip .03in}#1}

\newtheorem{theorem}{Theorem}[section]
\newtheorem{proposition}[theorem]{Proposition}
\newtheorem{lemma}[theorem]{Lemma}

\theoremstyle{definition}
\newtheorem{definition}[theorem]{Definition}
\newtheorem{example}[theorem]{Example}

\newtheorem{remark}[theorem]{Remark}

\numberwithin{equation}{section}

\allowdisplaybreaks[4]
\makeatletter \@addtoreset{equation}{section}

\makeatother

\begin{document}

\begin{center}
{\Large {\bf Zhu's Algebra of a $C_1$-cofinite Vertex Algebra}} \\
\vspace{0.5cm}

Lu Ding\\
Institute of Applied Mathematics,\\
Academy of Mathematics and Systems Science,\\
Chinese Academy of Sciences, Beijing 100190, China\\
Wei Jiang\\
Department of Mathematics,\\
Changshu Institute of Technology, Changshu 215500, China\\
Wei Zhang \\
School of Mathematics and Statistics  \\
Beijing Institute of Technology, Beijing 100080, China\\

\end{center}
\hspace{1.5 cm}
\begin{abstract}
For a $C_1$-cofinite vertex algebra $V$, we give an efficient way to calculate Zhu's algebra $A(V)$ of $V$ with respect to its $C_1$-generators and relations. We use two examples to explain how this method works.
\end{abstract}

\section{Introduction}
For a vertex algebra $V$, Zhu's algebra $A(V)$ is a powerful tool to study its representations \cite{Zhu1996}. But Zhu's algebra is hard to calculate directly. For example, in \cite{donlamtanyamyok2004,Wan1998}, representations of vertex algebras associated with $\mathcal{W}_3$ were studied for different central charges. In both papers, the generators of Zhu's algebras were given and some relations were calculated. In order to show the relations were enough, the authors used some extra algebraic structures to find some modules, then compared with the modules of the possible bigger algebras and got the Zhu's algebras. For some other vertex algebras like vertex algebras associated to even positive lattices \cite{don1993}, the representations were calculated without using Zhu's algebras at all since it seems very hard to find their Zhu's algebras. So how can we calculate Zhu's algebra directly? We think that the answer to this question is also important to the classification of rational vertex algebras.

In this paper, we give an method to solve this problem by using $C_1$-structures of vertex algebras. In \cite{KarLi1999}, the authors studied the generating spaces of a vertex algebra $V$ which play a role analogous to that of $\mathfrak{g}$ in $L(l,0)$ \cite{Frezhu1992}. If we pick a homogenous basis $\{u^1,\cdots, u^l\}$ of this generating space which is called a \emph{generating space of weak PBW type}\cite{KarLi1999}, we call these elements $u^1,\cdots, u^l$ \emph{$C_1$-generators}. In \cite{DesKac2005}, the authors say that the vertex algebra $V$ is \emph{strongly} generated by these elements.
In \cite{KarLi1999}, the authors showed that $V$ is spanned by
\begin{equation}\label{spanset}
u^{i_1}_{m_1}\cdots u^{i_k}_{m_k}\1,
\end{equation}
where $1\leq i_1,\cdots,i_k\leq l$, $k\in\Z_+$, $m_1\leq\cdots\leq m_k\leq -1$ and if $m_j=m_{j+1}$, then $i_j\geq i_{j+1}$.
So $u^i_mu^j$ can be written as the linear combinations of elements in (\ref{spanset}), for $1\leq i,j\leq l$ and $m\geq 0$. We call $u^i_mu^j=R(i,j,m)$ as \emph{$C_1$-relations} in \cite{Zha2014} where $R(i,j,m)$ is linear combinations of elements in (\ref{spanset}).  Furthermore we proved that we can construct a \emph{universal} vertex algebra $U$ from some abstract $C_1$-generators and $C_1$-relations, here universal means that any vertex algebra has the same $C_1$-generators and $C_1$-relations is a quotient vertex algebra of $U$. That means that there is another level of relations determined by an ideal $I$ of $U$, i.e. $V=U/I$.

We roughly explain what is $C_1$-structure. If $V$ is a $C_1$-cofinite vertex algebra \cite{KarLi1999}, $u^1,\cdots,u^l$ are its $C_1$-generators, then there are infinite many operators $u^i_m$ associated with $u^i$ for $1\leq i \leq l$ and $m\in\Z$. $V$ is just the perfect algebraic structure to express the relations among these operators. The most important relations are commutators $[u^i_m,u^j_n].$ But usually the commutators could not be given in the free associative algebra $A$ generated by these operators. The commutators could be given in the complete topological algebra $\bar{A}$ which is studied in \cite{Frezhu1992} and \cite{KarLi1999}. The commutator relations are $C_1$-relations. If we do not ask more relations, then we get a universal vertex algebra $U$. And $\mathcal{U}(U)$ equaling $\bar{A}$ modulo the commutator relations can be looked as the universal enveloping algebra of $U$. Then the category of admissible  $U$-modules is naturally equivalent to the category of $\Z_+$-graded $\mathcal{U}(U)$-modules.

Since the $C_1$-relations are commutator relations, naturally we ask whether the Jacobi identities
  $$[[u^i_l,u^j_m],u^k_n]+[[u^j_m,u^k_n],u^i_l]+[[u^k_n,u^i_l],u^j_m]=0$$
 hold or not in $\bar{A}$. If the Jacobi identities holds, then we have a PBW-like basis for $U$, otherwise we don't. We call the first kind of $C_1$-relations \emph{non-degenerate} and the other \emph{degenerate} \cite{DesKac2005,Zha2014}.

 We use $C_1$-structure to explain Zhu's algebra theory. $C_1$-relations give commutator relations among $u^i_m$ and make the space $T$ spanned by them look like a Lie algebra. So we can give a "triangular" decomposition of $T$ according to the weight of $u^i_m$:
\begin{equation}\label{tridec}
T=T_-\bigoplus T_0\bigoplus T_+,\end{equation}
where $T_-$ is spanned by $u^i_m$ for $1\leq i\leq l$ and $m>\wt u^i-1$, $T_0$ is spanned by $u^i_m$ for $1\leq i\leq l$ and $m=\wt u^i-1$, $T_+$ is spanned by $u^i_m$ for $1\leq i\leq l$ and $m<\wt u^i-1$. Any $T_0$-module $M$ which $T_-$ acts trivially on can be induced to a $T$-module. So "universal enveloping algebra" of $T_0$ is Zhu's algebra $A(U)$. If the $C_1$-relations are non-degenerate, then $A(U)$ has a PBW basis. If the $C_1$-relations are degenerate, then we need to find all other relations. For every element $v\in V$, the image of $v$ in $A(V)$ is just $v_{\wt v -1}$ which is written as a polynomial of $u^i_{\wt u^i-1}$.

For the second level of defining relations, i.e. an ideal $I$ of $U$,  we give a procedure to find $A(I)$. We can say that even we know the $C_1$-structure of a vertex algebra $V$, the calculation of Zhu's algebra $A(V)$ is still complex, but the procedure is quite straightforward, so we can use some computer program to do the calculations.

The paper is organized as following. In section 2 we analyse the $C_1$-structure of a vertex algebra. In section 3, we review Zhu' algebra theory and give an alternate way to do calculation for a $C_1$-cofinite vertex algebra. In section 4 and 5, we give our main results. Using Diamond Lemma, we give Zhu's algebra for non-degenerate and degenerate cases respectively. And we give two examples to explain our method.

\section{$C_1$-structure of a vertex algebra}

In this section, we review the abstract structure of a vertex algebra. In this paper we use standard notations as defined in \cite{Frelepmeu1988,Lepli2004} and all vector spaces are over complex field $\C$. Since the conformal vector $\w$ does not play any role in this paper, instead vertex operator algebras, We use $\Z_+$-graded vertex algebras defined as following:
\begin{definition} A \emph{$\Z_+$-graded vertex algebra} $V=(V,Y,\1)$ is
a $\Z_+$-graded vector space(graded by weights)
$$V=\bigoplus_{n\in \Z_+}V_n; \hskip .1in   {\rm for} \hskip .1in v\in V_n, \hskip .1in\wt{v}=n,$$
such that $$\dim V_0=1,$$
equipped with a  linear map $Y$ from $V$ to $({\rm
End}V)[[z,z^{-1}]]$, $$v\mapsto Y(v,z)=\sum_{n\in\Z}v_nz^{-n-1},
\hskip .1in $$where \hskip .1in $v_n\in \End V$
 and $u_mw\in V_{i+j-m-1}$ for $u\in V_i,\,\,w\in V_j$ and $m\in\Z$; and with a distinguished homogenous vector $\1\in V_0$ satisfying the following
conditions for $u,v\in V$:
$$Y({\bf 1},z)=\textrm{the identity operator on } V;$$
$$Y(u,z){\bf 1}\in V[[z]] \,\,\,\,\,and\,\,\, \,\, \lim_{z\rightarrow 0}Y(u,z){\bf 1}=u;$$
and the Jacobi identity:

\begin{eqnarray}\label{jacobiv}
&
&z_{0}^{-1}\delta\left(\frac{z_{1}-z_{2}}{z_{0}}\right)Y(u,z_{1})Y(v,z_{2})
-z_{0}^{-1}\delta\left(\frac{z_{2}-z_{1}}{-z_{0}}\right)Y(v,z_{2})Y(u,z_{1})
\nonumber\\
&
&=z_{2}^{-1}\delta\left(\frac{z_{1}-z_{0}}{z_{2}}\right)Y(Y(u,z_{0})v,z_{2}),
\end{eqnarray}
where $\delta(z)=\sum_{n\in\Z}z^n$ and  $(z_1+z_2)^n$ is defined
to be the formal power series
$\sum_{i=0}^{\infty}{n\choose i}z_1^{n-i}z_2^i$.  \\
\end{definition}
The definition of a vertex algebra is usually thought complex and hard to understand. We review the $C_1$-structure of vertex algebras \cite{KarLi1999,Zha2014}, and we will see that vertex algebras are not that difficult. The core of the definition, the Jacobi identity is equivalent to the following three formulas:
\begin{enumerate}
\item Commutator formula:

\begin{equation}\label{bracket}
[u_s,v_t]=u_sv_t-v_tu_s= \sum_{m\geq 0}{s\choose m}(u_mv)_{s+t-m};
\end{equation}

\item Iterate formula:

\begin{equation}\label{ass}
(u_nv)_m=\sum_{i\geq 0}{n\choose
i}\left((-1)^iu_{n-i}v_{m+i}-(-1)^{n+i}v_{n+m-i}u_i\right);
\end{equation}

\item Skew symmetry:  \begin{equation}\label{skewsym}
Y(u,z)v=e^{Dz}Y(v,-z)u,
\end{equation}
\end{enumerate}
where $u,v\in V$ and $s,t,m,n\in \Z$, $D$ is a linear map from $V$
to $V$ by mapping $v$ to $v_{-2}\1$.\\
We use the generators and their relations to see the roles that these formulas play. The following results can be found in \cite{KarLi1999,Zha2014}.

\begin{definition} For a vertex algebra $V$, set $V_+=\bigoplus_{n=1}^{\infty}V_n$. Define \emph{$C_1(V)$} to be the
subspace of $V$ linearly spanned by elements of type
$$u_{-1}v,\hskip .2in u_{-2}\1$$
for $u,v \in V_{+}$. If  $V/C_1(V)$  is finitely dimensional, then
we call $V$  \emph{$C_1$-cofinite}.

\end{definition}

\begin{proposition} Let
$U$ be a graded subspace of $V$ such that $V_{+}=U+C_1(V)$.
Then, as a vertex  algebra, $V$ is generated by $U$.
\end{proposition}
$U$ is called a generating space of weak $PBW$-type in \cite{KarLi1999}.
If we pick a homogeneous basis $\{u^1,u^2,\cdots\}$, then $V$ is
generated by this set, or $V$ is strongly generated by this set in \cite{DesKac2005}. We have the following definition as in \cite{Zha2014}.

\begin{definition} $u^1,u^2,\cdots$, the homogeneous basis vectors of $U$, are called \emph{$C_1$-generators.} \end{definition}

\begin{proposition}\label{span}
For a $C_1$-cofinite vertex algebra $V$ and $C_1$-generators $u^1,\cdots\,u^l$ of $V$, then $V$ is linearly spanned by elements
\begin{equation}\label{span1}
\ u^{i_1}_{n_1}\cdots u^{i_r}_{n_r}\1,
 \end{equation}
where $1\leq i_1,\cdots,i_r\leq l$; $n_1\leq \cdots \leq n_r<0$ and if $n_j=n_{j+1}$, then $i_j\leq i_{j+1}$.
\end{proposition}

With the same notations as the above proposition, for $1\leq i,j\leq l$ and $k\geq 0$
\begin{equation}\label{length1} u^i_ku^j=R(i,j,k),
\end{equation}
where $R(i,j,k)$ is a linear combination of elements in (\ref{span1}). Plug the above formula into (\ref{bracket}), we get for $m,n\in\Z$
\begin{equation}\label{length2}[u^i_m,u^j_n]=\sum_{k\geq 0}{m\choose k}R(i,j,k)_{m+n-k}.
\end{equation}
Now we define \emph{formal length} for an element $u^{i_1}_{n_1}\cdots u^{i_r}_{n_r}\1\in V$ as $\wt u^{i_1}+\cdots+ \wt u^{i_r}$ and for an operator $u^{i_1}_{n_1}\cdots u^{i_r}_{n_r}$ as $\wt u^{i_1}+\cdots+ \wt u^{i_r}$ for $1\leq i_1,\cdots,i_r\leq l$ and $n_1,\cdots,n_r\in \Z$. For the linear combination, we define the formal length as
the maximal one. We can see the formal length of left hand side of (\ref{length1}) is bigger than that of right hand side, and when applying (\ref{ass}) in (\ref{length2}) the formal length of left hand side of (\ref{length2}) is bigger than that of right hand side although the right hand side may have infinite summands(the formal length still makes sense).

\begin{remark}For a $C_1$-cofinite vertex algebras $V$ and its $C_1$-generators $u^1,\cdots,u^l$, we think that the essence of $V$ is the relations of these operators $u^i_m$ for $1\leq i\leq l$ and $m\in\Z$, and the structure of vertex algebras are the right algebraic structure to realize these relations which are complex and non-classical. What are these relations? We think that the most important relations are commutator relations $[u^i_m,u^j_n]$.
\end{remark}

From above argument, we see that $R(i,j,k)$ for $1\leq i,j\leq l$ and $k\geq 0$ determine the commutator relations of $u^i_m,u^j_n$ and the commutators have smaller formal lengthes. By the skew-symmetry (\ref{skewsym}), there are some redundances, so we have the following definition with the same notation above\cite{Zha2014}.

\begin{definition}
For a $C_1$-cofinite vertex algebra $V$ and its $C_1$-generators $u^1,\cdots,u^l$, we call $$u^i_ku^j=R(i,j,k)$$  \emph{$C_1$-relations} with respect to these generators for $1\leq i< j\leq l, 0\leq k$ or $1\leq i=j\leq l, 0<k$ being odd.
\end{definition}
\begin{remark}Using skew-symmetry(\ref{skewsym}), we can get $u^i_ku^j$ for all $1\leq i,j\leq l$ and $k\geq 0$, so we also call all these $u^i_ku^j=R(i,j,k)$  as $C_1$-relations with respect to these $C_1$-generators.
\end{remark}
For example, the vertex algebra associated to Virasoro algebras\cite{Frezhu1992}, has only one $C_1$-generator $\w$ with weight being $2$, and $C_1$-relations are ${\w}_1\w=2\w$ and ${\w}_3\w=\frac{c}{2}$. Using skew symmetry (\ref{skewsym}), we can get ${\w}_2\w=0$ and ${\w}_0\w={\w}_{-2}\1$.

\begin{remark}\label{topring}If the $C_1$-relations $R(i,j,k)$ only contain linear terms like the above example, then the space $T$ spanned by $u^i_m$ for $1\leq i\leq l$ and $m\in\Z$, with commutator relations gotten from $C_1$-relations (\ref{bracket}), forms a Lie algebra. The $C_1$-relations can be realized in the free associative algebra $A$ generated by $u^i_m$ as an ideal $C$ and we get the universal enveloping algebra $U(T)=A/C$. But if the $C_1$-relations $R(i,j,k)$ contains non-linear terms like the examples in section 4 and 5, then $T$ is not a Lie algebra, and the $C_1$-relations can not be realized in $A$. In \cite{Frezhu1992,KarLi1999}, the authors introduced a complete topological algebra $\bar{A}$ as follows. Since $\wt u^i_m=\wt u^i-m-1$, so $A$ is a $\Z$-graded algebra and $A=\bigoplus_{k\in\Z} A_k$ where $A_k$ is the subspace consisting of the elements of weight $k$. Set $A_n^k=\sum_{i\leq k}A_{n-i}A_i$, then $A_n^k\subset A_n^{k+1}$ and $\cap_{k\in\Z}A^k_n=0, \cup_{k\in\Z}A^k_n=A_n$. Hence $\{A_n^k|k\in\Z\}$ form a fundamental neighborhood system of $A_n$. Denote by $\bar{A}_n$ its completion; then the direct sum $\bar{A}=\bigoplus_{n\in\Z}A_n$ is a complete topological algebra. We can realize the $C_1$-relations in $\bar{A}$ i.e. an ideal of $\bar{A}$ generated by $[u^i_m,u^j_n]-\sum_{k\geq 0}{m\choose k}R(i,j,k)_{m+n-k}$.
\end{remark}

We can reverse this process to construct a vertex algebra abstractly \cite{Zha2014}. Suppose that $u^1,\cdots,u^l$ are some abstract symbols and each has a positive integer weight denoted by $\wt u^i$. Then associate $u^i$ a series of symbols $u^i_m$ for $m\in\Z$, which also has weight as $\wt u^i_m=\wt u^i-m-1$. Let $A$ be the associative algebra generated by $u^i_m$ for $1\leq i\leq l$ and $m\in\Z$, then $A$ is $\Z$-graded. Let $B$ be the associative subalgebra of $A$ generated by $u^i_m$ for $1\leq i\leq l$ and $m<0$, then $B$ is $\Z_+$-graded. Let $A$ act on a vacuum vector $\1$ freely, we get a $\Z$-graded $A$-module $\bar{U}$ by setting the weight of $\1$ being $0$. Let $\bar{U'}$ be the subspace of $\bar{U}$ which is free $B$ module generated by $\1$. It is clear $\bar{U'}$ is $\Z_+$-graded. Now let $u^i=u^i_{-1}\1$ and associate $u^i$ a vertex operator $Y(u^i,z)=\sum _{m\in\Z}u^i_mz^{-m-1}$, and $Y(\1,z)=\id$ . Then use (\ref{ass}), we can associate every monomial $u$ in $\bar{U}$ a vertex operator $Y(u,z)=\sum  _{m\in\Z}u_mz^{-m-1}$, then linearly extend to the whole $\bar{U}$.

Now we can describe $C_1$-relations for these generators. Let $R(i,j,k)$ be any element in $\bar{U'}$, with weight being $\wt(u^i)+\wt(u^j)-k-1$ for $i\leq j$, $m\geq 0$ or $i=j$ $m\geq 0$ being odd. Then $u^i_ku^j=R(i,j,k)$ are $C_1$-relations for our generators.

Let $\bar{O}$ be a $A$-submodule of $\bar{U}$ generated by $\bar{U}_n$ for $n<0$ and $u^i_m\1$ for $m\geq 0$. Then all the operators $u_m$ defined above for $u\in \bar{U}$ make sense when acting on $U'=\bar{U}/\bar{O}$.  Apply the $C_1$-relations into $U'$.  Let $O$ be $A$-submodule of $U'$ generated by
$$([u^i_m,u^j_n]-\sum_{k\in\Z_+}{m\choose k}R(i,j,k)_{m+n-k})w,$$
where $w\in U', m,n\in\Z,1\leq i,j\leq l.$ Let $U=U'/O$, we have the following theorem in \cite{Zha2014}.
\begin{theorem}\label{abconva} $U$ is a vertex algebra with $C_1$-generators $u^1,\cdots, u^l$, satisfying the given $C_1$-relations.
\end{theorem}
$U$ is a universal vertex algebra, i.e. any vertex algebra which has the given $C_1$-generators and $C_1$-relations must be a quotient of $U$ modulo an ideal $I$. So $I$ is another level of relations which define the structure of vertex algebra $V$ with respect to the generators.

For the example above, if we start with one $C_1$-generator $\w$ and $C_1$-relations ${\w}_1\w=2\w$ and ${\w}_3\w=\frac{c}{2}$, we will get a universal vertex algebra $U$.  For $c=c_{p,q}=1-\frac{6(p-q)^2}{pq}$, where $p,q\in\{2,3,4,\cdots\}$, the vertex algebra $U$ has an ideal, and we get the minimal sires Virasoro vertex algebras\cite{Wan1993}.

$C_1$-relations determine the commutator relations among $u^i_m$ in $A$ or $\bar{A}$ for $1\leq i\leq l$ and $m\in \Z$, and these relations are enough to define a universal vertex algebra $U$. There is a natural question, Jacobi identities
\begin{equation}\label{jacobi1}[[u^i_l,u^j_m],u^k_n]+[[u^j_m,u^k_n],u^i_l]+[[u^k_n,u^i_l],u^j_m]=0\end{equation}
hold or fail? Hence there are two different kinds of $C_1$-relations, one is \emph{non-degenerate} if Jacobi identities hold, another is \emph{degenerate} if Jacobi identities fail.

In section 4 and 5, we will give two more examples of vertex algebras which are generated abstractly.
\section{Zhu's algebra theory}

In this section we review Zhu's algebra theory and study the relation between Zhu's algebra and $C_1$-structure. Let $V=(V,Y,1)$ be a $Z_+$-graded vertex algebra. In this paper, we only consider admissible $V$-module category.  An admissible $V$-module $M$ is a $\Z_+$-graded vector space
$$M=\bigoplus_{n\in\Z_+}M(n)$$
equipped with a linear map
\begin{eqnarray*}
& V & \rightarrow (\End M)[[z^{-1},z]],\\
& v  & \mapsto Y_M(v,z)=\sum_{n\in\Z}v_nz^{-n-1}\,(v_n\in\End M) for \,\,v\in V
\end{eqnarray*}
satisfying the following conditions for $u,v\in V,w\in M$:
$$Y_M(\1,z)=1;$$
\begin{eqnarray}\label{mjacobiv}
&
&z_{0}^{-1}\delta\left(\frac{z_{1}-z_{2}}{z_{0}}\right)Y_M(u,z_{1})Y_M(v,z_{2})
-z_{0}^{-1}\delta\left(\frac{z_{2}-z_{1}}{-z_{0}}\right)Y_M(v,z_{2})Y_M(u,z_{1})
\nonumber\\
&
&=z_{2}^{-1}\delta\left(\frac{z_{1}-z_{0}}{z_{2}}\right)Y_M(Y(u,z_{0})v,z_{2});
\end{eqnarray}
$$u_mM(n)\subseteqq M(r+n-m-1),$$ for $u\in V_r$.

Now we review Zhu's algebra theory \cite{Zhu1996}. The idea of Zhu's algebra is to study $M_0$ instead of $M$, since $M_0$ is quite smaller. Let $O(V)$ denote the linear span of element $u\circ v$ where for homogeneous $u\in V$ and $v\in V$,
$$u\circ v=\Res _z Y(u,z)v\frac{(1+z)^{\wt u}}{z^2}.$$
Define the linear space $A(V)$ to be the quotient $V/O(V)$. Define another product $\ast $ on $V$ for homogeneous $u\in V$ and $v\in V$,
\begin{equation}\label{zhuast} u\ast v=\Res _z Y(u,z)v\frac{(1+z)^{\wt u}}{z}.\end{equation}
and extend linearly on $V$.The multiplication $\ast$ induces the multiplication on the quotient $A(V)$ and is associative in $A(V)$. The image of the vacuum $\1$ in $A(V)$ becomes the identity element.

Denote by $o(a)=a_{\wt a-1}$ for homogeneous $a\in V$ and extends to the whole $V$ by linearity. For an admissible $V$-modules $M=\oplus_{n=0}^{\infty}M_n$, $o(a)$ preserves $M_0$ and we have the following theorems.
\begin{theorem}\label{zhuthe1}The identities
$$o(a)o(b)=o(a\ast b)$$
and $$o(c)=0$$
hold in $\End(M_0)$ for every $a,b\in V$ and $c\in O(V)$. Thus, the top level $M_0$ is an $A(V)$-module.
\end{theorem}

By above theorem, we can also write the image of $a$ in $A(V)$ as $o(a)$ for $a\in V$.
\begin{theorem}There is a one-to-one correspondence between the set of irreducible $A(V)$-modules and the set of irreducible $V$-modules.
\end{theorem}
\begin{proposition}\label{zhuideal}
If $I$ is an ideal of a vertex algebra $V$, then $A(I)$ the image of $I$ in $A(V)$, is a two-sided ideal of $A(V)$, and $A(V/I)$ is isomorphic to $A(V)/A(I)$.
\end{proposition}

Zhu'algebra is a quite powerful tool to study representations of a vertex algebra, but it is very difficulty to calculate from its definition.  Now we study the relations between Zhu's algebra and $C_1$-structures of a $C_1$-cofinite vertex algebra $V$ with $C_1$-generators $u^1,u^2,\cdots,u^l$. $A(V)$ has the same generators as $C_1$-generators.
\begin{proposition}\label{avgenerator}For a $C_1$-cofinite vertex algebra $V$ and $C_1$-generators $u^1,\cdots\,u^l$ of $V$, $A(V)$ is generated by $o(u^1),o(u^2),\cdots,o(u^l)$. 
\end{proposition}
\begin{proof} By Proposition \ref{span}, $V$ is spanned by elements $u^{i_1}_{n_1}\cdots u^{i_r}_{n_r}\1$ where $1\leq i_1,\cdots,i_r\leq l$; $n_1\leq \cdots \leq n_r<0$ and if $n_j=n_{j+1}$, then $i_j\leq i_{j+1}$.
We can rewrite $u^{i_1}_{n_1}\cdots u^{i_r}_{n_r}\1$ as $(u^{i_1}_{n_1}\1)_{-1}u^{i_2}_{n_2}\cdots u^{i_r}_{n_r}\1$. By (\ref{zhuast}) and Theorem \ref{zhuthe1},
\begin{align}\label{nb1}
o(u^{i_1}_{n_1}\cdots u^{i_r}_{n_r}\1)&=o((u^{i_1}_{n_1}\1)_{-1}u^{i_2}_{n_2}\cdots u^{i_r}_{n_r}\1)\notag\\
&=o({\Res} _zY(u^{i_1}_{n_1}\1,z)u^{i_2}_{n_2}\cdots u^{i_r}_{n_r}\1\frac{(1+z)^{\wt u^{i_1}-n_1-1}}{z})\notag\\
&-o(\sum_{k\geq 1}{\wt u^{i_1}-n_1-1 \choose k}(u^{i_1}_{n_1}\1)_{k-1}u^{i_2}_{n_2}\cdots u^{i_r}_{n_r}\1)\notag\\
&=o(u^{i_1}_{n_1}\1)o(u^{i_2}_{n_2}\cdots u^{i_r}_{n_r}\1)\notag\\
&-o(\sum_{k\geq 1}{\wt u^{i_1}-n_1-1 \choose k}(u^{i_1}_{n_1}\1)_{k-1}u^{i_2}_{n_2}\cdots u^{i_r}_{n_r}\1)\notag\\
&=(-1)^{\wt u^{i_1}+n_1}{n_1\choose \wt u^{i_1}-1}o(u^{i_1})o(u^{i_2}_{n_2}\cdots u^{i_r}_{n_r}\1)\notag\\
&-o(\sum_{k\geq 1}{\wt u^{i_1}-n_1-1 \choose k}(u^{i_1}_{n_1}\1)_{k-1}u^{i_2}_{n_2}\cdots u^{i_r}_{n_r}\1).
\end{align}
The summands of (\ref{nb1}) have smaller formal lengths, hence by induction on the formal length, we show that $A(V)$ is generated by $o(u^1),\cdots,o(u^l)$.\end{proof}

For a homogeneous vector $u\in V$, $o(u)$ can be written as sum of products of $o(u^1),\cdots,$\\$o(u^l)$ by above proposition. By (\ref{ass}),  $o(u)=u_{\wt u-1}$ can be expanded as sum of products of $u^i_m$ for $1\leq i\leq l$ and $m\in\Z$. Since $o(u)$ acts on $M_0$,  we can omit summands of $u_{\wt u-1}$ which are $0$ when acting on $M_0$. Clearly, there are only finitely many summands of $u_{\wt u-1}$ left. For a summand $u^{i_1}_{m_1}u^{i_2}_{m_2}\cdots u^{i_k}_{m_k}$, if $\wt u^{i_j}_{m_j}<\wt u^{i_{j+1}}_{m_{j+1}}$, then we can switch their orders by (\ref{bracket}) and $C_1$-relations. We can keep doing the above procedure until $u_{\wt u-1}$ is written as sum of products of $o(u^i)$. This give another way to calculate the image $o(u)$ of $u$ in $A(V)$ which is different to the method of definition of $A(V)$.

By the above proposition, to calculate $A(V)$ becomes finding all the relations among $o(u^1),\cdots,o(u^l) $.
In last section, we analyze the relations for the $C_1$-generators, and divide the total relations to two level, the first one is $C_1$-relations which will give a universal vertex algebra $U$, and the second one is some ideal $I$ of $U$. Hence we first calculate $A(U)$, then use Proposition \ref{zhuideal} to calculate $A(V)$.

Now we roughly introduce our idea to calculate $A(U)$. A representation of $U$ can be looked as a representation of  $u^i_m$ for $1\leq i\leq l$ and $m\in\Z$ which satisfy $C_1$-relations. $C_1$-relations can be looked as commutator relations among these operators, and we look the space spanned by these operators with their commutator relations as a Lie algebra $T$. Hence like infinite-dimensional Lie algebras, we can give a triangular decomposition of these operator, $$T=T_-\oplus T_0\oplus T_+,$$ where $T_-$ is spanned by negative weight operators, $T_0$ is spanned by weight $0$ operators and $T_+$ is spanned by positive weight operators. Like highest weight modules for Lie algebras, we can construct a trivial module for $T_-$, then induce it to a module for $T_-\oplus T_0$ and induce it to a module for $T$. Rigorously in mathematics, we can realize above idea as follows.  Let $A$ be the free associative algebra generated by $u^i_m$ for $1\leq i\leq l$ and $m\in\Z$. Obviously, $A$ is $\Z$-graded by weights. Now we can construct a free $A$-module $\tilde{M}$ generated by $\1$. Then $\tilde{M}$ is also $\Z$-graded if we set the weight of $\1$ as $0$. Let $N$ be the submodule of $\tilde{M}$ generated by negative weight elements in $\tilde{M}$, then we get a $A$-module $\bar{M}=\tilde{M}/N$ which is $\Z_+$-graded. Now we apply the $C_1$-relations in $\bar{M}$:
\begin{equation}
M=\bar{M}/R,
\end{equation}
where $R$ is a $A$-submodule of $\bar{M}$ generated by $([u^i_m,u^j_n]-\sum_{k\in\Z_+}{m\choose k}R(i,j,k)_{m+n-k})w$
for $w\in \bar{M}, m,n\in\Z,1\leq i,j\leq l.$ Since $R$ is homogeneous, $M$ is $\Z_+$-graded and is an admissible $U$-module. Now $M_0$ is an $A(U)$-module, and in our construction of $M$, we do not add any other relations, so $M_0$ is a free $A(V)$ module with only one generator $\1$, and as vector spaces $A(V)$ is isomorphic to $M_0$. So our question is to describe $M_0$. We will discuss it for degenerate case and non-degenerate case respectively. And we also will handle the second relations caused by an ideal of $U$.

\section{Non-degenerate Case}

Let $V$ be a $C_1$-cofinite vertex algebra generated by $C_1$-generators $u^1, u^2,\cdots, u^l$ with $\wt u^i>0.$ Suppose that $\wt u^i\geq \wt u^{i+1}$. These generators satisfy $C_1$-relations $u^i_mu^j=R(i,j,m)$, for $1\leq i, j\leq l,m\geq 0$. Suppose that these relations are non-degenerate, by theorem \ref{abconva}, we get a universal vertex algebra $U.$ $V$ is a quotient vertex algebra of $U$ by modulo an ideal $I$, which is usually generated by some singular vectors.

In section 2, we give the construction of $U$. The linear space $U'$ is quite clear and the following elements $$u^{i_1}_{n_1}\cdots u^{i_k}_{n_k}\1+\bar{O},$$
where $n_k<0$ and $\wt u^{i_j}_{n_j}\cdots u^{i_k}_{n_k}\1\geq 0$ for $1\leq j\leq k$ form a basis of $U'$. Apply the $C_1$-relations into $U'$.  Let $O$ be $A$-submodule of $U'$ generated by
$$([u^i_m,u^j_n]-\sum_{k\in\Z_+}{m\choose k}R(i,j,k)_{m+n-k})w,$$
where $w\in U', m,n\in\Z,1\leq i,j\leq l.$  $U=U'/O$. We use Diamond Lemma\cite{Ber1978} to get a basis of $U$.

Let $S$ be a set of pairs of the form $\sigma =(W_\sigma,f_\sigma)$, where $W_\sigma=u^i_{m}u^j_{n}\in A$ for $0>m>n,$ or $0>m=n$ and $i>j$, or $m\geq 0>n$, or $m,n\geq 0$ and $\wt u^i-m< \wt u^j-n$, or $m,n\geq 0$ and $\wt u^i-m= \wt u^j-n$ and $i>j$, and $f_\sigma=u^j_{n}u^i_m+\sum_{k\in\Z_+}{m\choose k}R(i,j,k)_{m+n-k}$. For any $\sigma\in S$ and $x,y\in A$ being monomials, let $r_{x\sigma y}$ denote linear morphism of $U'$ that fixes all basis elements of $U'$ other than $xW_\sigma y\1+\bar{O}$ and sends it to $xf_\sigma y\1+\bar{O}$. The map is well defined since the image of $\bar{O}$ belongs to $\bar{O}$. $S$ is called a \emph{reduction system}, and the linear maps $r_{x\sigma y}$ are called \emph{reductions}.

We say a reduction $r_{x\sigma y}$ acts \emph{trivially} on an element $a\in U'$ if the coefficient of $xW_\sigma y\1$ in $a$ is zero, and we shall call $a$ \emph{irreducible} (under $S$), if every reduction is trivial on $a$, i.e., if $a$ involves none of the monomials $xW_\sigma y\1$. The vector space of all irreducible elements of $U'$ will be denoted $U'_{irr}$.

Define a total ordering of $u^i_m$ by setting $u^i_m <u^j_n $ if $0>m>n,$ or $0>m=n$ and $i>j$, or $m\geq 0>n$, or $m,n\geq 0$ and $\wt u^i-m< \wt u^j-n$, or $m,n\geq 0$ and $\wt u^i-m= \wt u^j-n$ and $i>j$. Let us partially order $U'$ by setting $u<v$ if $\wt (u)<\wt (v)$ , or if $\wt(u)=\wt(v)$ and $u$ is of smaller formal length than $v$, or if $u$ is a permutation of the terms of $v$ but has less misorderings with respect to the total ordering defined above. It is obvious that $<$ is a semigroup partial ordering of $U'$ and has descending chain condition, i.e. there are only finitely many base elements $<$ any given element.

For any $\sigma\in S$, $f_\sigma<W_\sigma$, and we say $<$ is \emph{compatible} with $S$. Since $<$ has descending chain condition and is compatible with $S$, every element $w\in U'$ can be reduced to irreducible elements. We shall call an element $w\in U'$ by \emph{reduction-unique} if it is reduced to a common value, and which is denoted by $r_S(w)$.

Now we state the Diamond Lemma for our situation.
\begin{lemma}{\rm \textbf{(Diamond Lemma)}} $$U\cong U'_{irr}.$$
\end{lemma}
\begin{proof}First we show all elements of $U'$ are reduction-unique under $S$. It will be sufficient to prove all monomials $w\in U'$ reduction-unique. Since there is a semigroup partial ordering $<$ on $U'$ which is compatible with $S$, and has descending chain condition, we may assume inductively that all monomials $<w$ are reduction-unique. Now if two reductions $r_\sigma$ and $r_\tau$ act nontrivially on $w$, and $r_\sigma(w)\neq r_\tau(w)$, we will show $r_S(r_\sigma(w))=r_S(r_\tau(w))$. There are two different cases, according to the relative places of $W_\sigma$ and $W_\tau$ in $w$.

Case 1. $W_\sigma$ and $W_\tau$ overlap in $w$. $w$=$xu^i_su^j_mu^k_ny\1$  for $x,y\in A$ and $W_\sigma=u^i_su^j_m, W_\tau=u^j_mu^k_n,$ $W_\mu=u^i_su^k_n$. $r_\tau(r_\mu (r_\sigma(w)))=r_\tau(r_\mu(xu^j_mu^i_su^k_ny\1+xf_\sigma u^k_ny\1))=r_\tau(xu^j_mu^k_nu^i_sy\1+xu^j_mf_\mu y\1+xf_\sigma u^k_ny\1)=xu^k_nu^j_mu^i_sy\1+xf_\tau u^i_sy\1+xu^j_mf_\mu y\1+xf_\sigma u^k_ny\1$ and $r_\sigma(r_\mu (r_\tau(w)))=r_\sigma(r_\mu(xu^i_su^k_nu^j_my\1+xu^i_sf_\tau y\1))=r_\sigma(xu^k_nu^i_su^j_my\1+xf_\mu u^j_my\1+xu^i_sf_\tau y\1)=xu^k_nu^j_mu^i_sy\1+xu^i_sf_\tau y\1+xf_\mu u^j_my\1+xu^k_nf_\sigma y\1.$  Since $C_1$-relations are non-degenerate, $r_\tau(r_\mu (r_\sigma(w)))-r_\sigma(r_\mu (r_\tau(w)))=0.$  By assumption, $r_\tau(w)$ and $r_\sigma(w)$ are reduction-unique, so they reduce to the same element.

Case 2. $W_\sigma$ and $W_\tau$ are disjoint in $w$. Then $w=xW_\sigma yW_\tau z\1$ where $x,y,z\in A$. It is clear that $r_\sigma(r_\tau(w))=r_\tau (r_\sigma(w))$. By assumption, $r_\tau(w)$ and $r_\sigma(w)$ reduce to the same element.

So we show that all elements of $U'$ are reduction-unique under $S$. It is clear that $r_S(O)=0$. So $U'=U'_{irr}\oplus O$ and $U\cong U'_{irr} $.
\end{proof}

By the result of \cite{Zha2014}, $U$ is the university vertex algebra. So we have the following proposition. A similar result was given in \cite{DesKac2005}.
\begin{proposition} For given $C_1$-generators $u^1, u^2,\cdots, u^l$ and their $C_1$-relations $R(i,j,m)$, if these $C_1$-relations are non-degenerate, then there exists a universal vertex algebra $U$, with the following vectors being a basis of $U$
$$u^{i_1}_{-n_1}\cdots u^{i_k}_{-n_k}\1, $$
where $n_1\geq \cdots \geq n_k\geq 1$ and $i_j\leq i_{j+1}$ if $n_j=n_{j+1}$.
\end{proposition}

Now we calculate Zhu'algebra $A(U)$ of $U$, it is generated by $o(u^1),\cdots,o(u^l)$  by Proposition \ref{avgenerator}. These generators satisfy the obvious relations coming from the $C_1$-relations \begin{equation}\label{c1relationforzhu}
[o(u^i),o(u^j)]=\sum_{k\geq 0}{\wt (u^i)-1\choose k}o(R(i,j,k)).
\end{equation}
Hence $A(U)$ is linearly spanned by all elements
\begin{equation}\label{4basis1}
o(u^{i_1})\cdots o(u^{i_k}){\bf 1},
\end{equation} here $k\in\Z_+$ and $1\leq i_1\leq \cdots \leq i_{k}\leq l$. In fact these products are linearly independent.
\begin{proposition}\label{au}$A(U)$ has a basis consisting of the above elements.\end{proposition}
\begin{proof} First we construct an admissible $U$-module. Use the notations above. $\bar{N}$ be a free $A$-module with generator $\1$(here we abuse the symbol $\1$). Set the weight of $\1$ to be $0$, then $\bar{N}$ is $\Z$-graded $A$-module and $\bar{N}=\bigoplus_{k=-\infty}^{\infty}\bar{N}_k$. Let $\bar{N}'$ be the submodule of $\bar{N}$ generated by $\bar{N}_k$ for $k<0$. Let $\tilde{N}=\bar{N}/\bar{N}'$, so $\tilde{N}$ is a $\Z_+$-graded $A$-module.
Apply $C_1$-relations into $\tilde{N}$. Let $O$ be $A$-submodule of $\tilde{N}$ generated by $$([u^i_m,u^j_n]-\sum_{k\in\Z_+}{m\choose k}R(i,j,k)_{m+n-k})w$$ where $w\in \tilde{N}, m,n\in\Z,1\leq i,j\leq l.$
Let $N=\tilde{N}/O$. $N$ is a $Z_+$-graded $A$-module,$N=\bigoplus_{i\geq 0}N_i$, and the operator $u^i_m$ satisfy $C_1$-relations. Hence $N$ is an admissible $U$-module with one generator $\1$. We can use the Diamond Lemma and the same argument as Lemma 4.1 to find a basis of $N_0$ as following
\begin{equation}\label{4basis2}
u^{i_1}_{\wt u^{i_1}-1}\cdots u^{i_k}_{\wt u^{i_k}-1}\1,
\end{equation} here $k\in\Z_+$ and $1\leq i_1\leq \cdots \leq i_{k}\leq l$.

Since $N$ is an admissible $U$-module and generated by $\1$, $N_0$ is an $A(U)$-module and generated by $\1$. Hence $N_0$ is a quotient space of $A(U)$, compare (\ref{4basis1}) and (\ref{4basis2}), we get a basis of $A(U)$ consisting of $o(u^{i_1})\cdots o(u^{i_k})\1,$ where $k\in\Z_+$ and $1\leq i_1\leq \cdots\leq i_k\leq l$.
\end{proof}

Now we consider $A(V)$. Let $V=U/I$ be a quotient vertex algebra. Usually we describe $I$ by giving some singular vectors $a^1,\cdots, a^k$ of $U$, and $I$ is spanned by \begin{equation}\label{idealele}u^{i_1}_{n_1}\cdots u^{i_r}_{n_r}a^j_{-s}\1,\end{equation} where $1\leq i_1,\cdots,i_r\leq l$, $n_1\leq \cdots \leq n_r$, $1\leq j\leq k$ and $s\geq 1$. By Proposition \ref{zhuideal}, $A(V)=A(U)/A(I)$, hence we need to describe $A(I)$. $A(U)$ is generated by $o(u^i)$ for $1\leq i\leq l$, which satisfy $C_1$-relations (\ref{c1relationforzhu}). Hence we need to find the generators of $A(I)$ as $A(U)$'s ideal.

\begin{lemma}\label{lemmacalzhu}$o(u_{-m}v)=x_1o(u)o(v)+x_2o(u_0v)+x_3o(u_1v)+\cdots$ and $o(u_{-s}\1)=yo(u)$, where $u,v\in U$, $m>0$, and $x_i, y\in \C$.
\end{lemma}
The proof of above lemma is quite easy by the definition of Zhu's algebra or the method introduced after Proposition \ref{avgenerator}. By above lemma, we get $A(I)$ as ideal of $A(U)$ is generated by
\begin{equation}\label{generatorofai}
u^{i_1}_{n_1}\cdots u^{i_r}_{n_r}a^j,
\end{equation}
where $1\leq i_1,\cdots,i_r\leq l$, $0\leq n_1\leq \cdots \leq n_r$ and $1\leq j\leq k$.

It seems that there are infinitely many calculations, but by the following example we can see that the calculations will stop somewhere and this method is quite efficient.

\begin{example}We consider the vertex algebra related to Zamolodchikov $\mathcal{W}_3$ algebra \cite{Zam1985}, which is the simplest example of $\mathcal{W}$-algebras. Let the central charge be $-2$. In \cite{Wan1998}, the author had studied the representations of this vertex algebra by calculating Zhu's algebra, but he used some extra algebraic structure related to $\mathcal{W}_3$. Here we only use the abstract structure to calculate Zhu's algebra. The vertex algebra is generated by two $C_1$-generators $\omega,v$ with weight of $\omega$ being $2$ and weight of $v$ being $3$. Their $C_1$-relations are given by the following:
\begin{eqnarray*}
& &\w_1\w=2\w,\\
& &\w_3\w=-\1,\\
& &\w_0v=v_{-2}\1,\\
& &\w_1v=3v,\\
& &v_1v=\frac{8}{3}\w_{-1}\w-\w_{-3}\1,\\
& &v_3v=2\w,\\
& &v_5v=-\frac{2}{3}\1.\\
\end{eqnarray*}
Using skew symmetry formula, we can get
\begin{eqnarray*}
& &\w_0\w=\w_{-2}\1,\\
& &\w_2\w=0,\\
& &v_0v=\frac{8}{3}\w_{-2}\w-\frac{2}{3}\w_{-4}\1,\\
& &v_2v=\w_{-2}\1,\\
\end{eqnarray*}
$\widetilde{\mathcal{W}_3}$ is the universal vertex algebra generated by $\w$ and $v$. Direct calculations show that these $C_1$-relations are non-degenerate. Its Zhu's algebra $A(\widetilde{\mathcal{W}_3})$ is generated by $o(\w)$ and $o(v)$, which satisfy the relation $$[o(\w),o(v)]=[\w_1,v_2]=0,$$ so $A(\widetilde{\mathcal{W}_3})$ is the polynomial algebra $\C[o(\w),o(v)]$.

In \cite{Wan1998}, the author studied a quotient vertex operator algebra $\mathcal{W}_{3,-2}$ of $\widetilde{\mathcal{W}_3}$.  $\widetilde{\mathcal{W}_3}$ has an ideal $I$ generated by a singular vector:
$$v_s=\frac{3}{2}v_{-1}v-\frac{19}{36}\w_{-2}\w_{-2}\1-\frac{8}{9}\w_{-1}\w_{-1}\w-\frac{14}{9}\w_{-1}\w_{-3}\1+\frac{44}{9}\w_{-5}\1,$$
and $\mathcal{W}_{3,-2}=\widetilde{\mathcal{W}_3}/I$.

By Lemma \ref{lemmacalzhu}, to calculate $A(I)$ we need to calculate $o(v_{n_1}\cdots v_{n_l}\w_{m_1}\cdots \w_{m_k}v_s)$, where $0\leq n_1\leq \cdots\leq n_l$ and $0\leq m_1\leq \cdots\leq m_k$. Since we look $A(I)$ as an ideal of $A(\widetilde{\mathcal{W}_3})$, we only need to find the generators of $A(I)$. Now we have one generator $o(v_s)$. $\w_mv_s=0$ for $m\geq 2$. $\w _1v_s=6v_s$, so there is no necessity to calculate $\w_1\w_1 v_s$ and so on,  and we do not get a new generator of $A(I)$. $\w_0{v_s}={v_s} _{-2}\1$, by Lemma \ref{lemmacalzhu}, we do not get a new generator of $A(I)$ and there is no necessity to calculate $\w_0\w_0v_s$ and so on since by (\ref{bracket}) and Lemma \ref{lemmacalzhu}, these vectors will not give more new generators of $A(I)$. $v_mv_s=0$ for $m\geq 3$.
Let $$v_s'=\frac{9}{2}v_{-4}\1+9\w_{-2}v-6\w_{-1}v_{-2}\1.$$
By direct calculations, we get
$$v_2v_s=\frac{98}{27}v_s',$$
$$v_2v_s'=36v_s.$$
So we get a new generator $o(v_s')$ for $A(I)$. Now we need to handle $v_1v_s$ and $v_1v_s'$.
By direct calculations, we have
$$v_1v_s=\frac{49}{54}{v_s'}_{-2}\1 ,$$
$$v_1v_s'=9{v_s}_{-2}\1.$$
Hence there is no new generator and we need to handle $v_0v_s$ and $v_0v_s'$.

$$v_0v_s=\frac{8}{9}\w_{-1}v_s'+\frac{2}{27}{v_s'}_{-3}\1 ,$$
$$v_0v_s'=12{v_s}_{-3}\1-10(3v_{-2}v_{-2}\1+8\w_{-7}\1-8\w_{-5}\w-2\w_{-3}\w_{-3}\1-4\w_{-4}\w_{-2}\1-4\w_{-2}\w_{-2}\w).$$
By Lemma \ref{lemmacalzhu}, $o(v_0v_s)$ is not a new generator of $A(I)$ and $o(v_0v_s')$ is new.
\begin{align*}
&v_0v_0v_s'\\
=&12(v_0v_s)_{-3}\1\\
&+10(168v_{-8}1+96\w_{-6}v+32\w_{-5}v_{-2}\1+16\w_{-4}v_{-3}\1-24\w_{-3}v_{-4}\1+32\w_{-2}v_{-5}\1 \\
=&12(v_0v_s)_{-3}\1+\frac{32}{3}{v_s'}_{-5}\1+\frac{80}{3}\w_{-2}{v_s'}_2\1-\frac{160}{3}\w_{-3}v_s'.
\end{align*}
The calculations can stop here.

By direct calculation, we get $$o(v_s)=\frac{3}{2}o(v)^2-\frac{1}{9}o(\w)^2(8o(\w)+1),$$
$$o(v'_s)=0,$$
$$o(v_0v_s')=-18o(v)^2+\frac{4}{3}o(\w)^2(8o(\w)+1),$$

so we get $A(\mathcal{W}_{3,-2})=\C[o(\w),o(v)]/<\frac{3}{2}o(v)^2-\frac{1}{9}o(\w)^2(8o(\w)+1)>.$

\end{example}

\begin{remark}For general situations, we want to find the generators of $A(I)$ as ideal of $A(U)$. At beginning, we have generators $g^1=o(a^1),\cdots,g^k=o(a^k)$,  by Lemma \ref{lemmacalzhu}, we calculate the element in (\ref{generatorofai}) for $r=1,2,\cdots$. For $r=1$, if the element can be written as $u_{-n}g^i_{-m}\1$ or their sum, then this element is not a new generator, if not, then we have a new generator $g^{k+1}$. We keep doing this process until we could not get any new generators. We believe that there are only finite many generators like above example. The process looks complex but it is quite straightforward and efficient.
\end{remark}

\section{Degenerate Case}

In this section, we deal with degenerate cases. We use the same notations as last section. Let $V$ be a $C_1$-cofinite vertex algebra generated by $C_1$-generators $u^1, u^2,\cdots, u^l$ with $\wt u^i>0.$ These generators satisfy $C_1$-relations $u^i_mu^j=R(i,j,m)$, for $1\leq i, j\leq l,m\geq 0$. Suppose that these relations are degenerate, by theorem \ref{abconva}, we get a universal vertex algebra $U.$ $V$ is a quotient vertex algebra of $U$ by modulo an ideal $I$, which is usually generated by some singular vectors.

Since $C_1$-relations are degenerate and Jacobi identities (\ref{jacobi1}) fail, $U$ does not have a PBW basis like non-degenerate case (\ref{span}). We first analyze the linear relations among these vectors. We still use the reduction system. There are maybe two ways to reduce an element $u\in U'$, $r_{\sigma_n}\cdots r_{\sigma_1}(u)$ and $r_{\tau_m}\cdots r_{\tau_1}(u)$, to irreducible elements in $U'$, and their difference $r_{\sigma_n}\cdots r_{\sigma_1}(u)-r_{\tau_m}\cdots r_{\tau_1}(u)$ will give a linear relation in $U'_{irr}$. It is clear that all the linear relations in $U'_{irr}$ are caused by the failure of the uniqueness of reduction. We will show that all the linear relations in $U'_{irr}$ can be written as linear combinations of any irreducible images of the following element in $U'$ under reductions
\begin{equation}\label{reduction1}
x([[u^i_s,u^j_m],u^k_n]+[[u^j_m,u^k_n],u^i_s]+[[u^k_n,u^i_s],u^j_m])y\1,
\end{equation} where $1\leq i,j,k\leq l$, $s,m,n\in\Z$ and $x,y \in A$.

Since $U'$ is homogeneous and any reduction keeps the weight, we only need to consider the homogenous $u
\in U'$. If two reductions $r_\sigma$ and $r_\tau$ act nontrivially on $u$, and $r_\sigma (u)\neq r_\tau (u)$, then $r_\sigma (u)$ and $r_\tau (u)$ have smaller formal lengths than that of $u$. So there are two cases according to the relative places of $W_\sigma$ and $W_\tau$ in $u$.

Case 1. $W_\sigma$ and $W_\tau$ overlap in $u$. $u$=$xu^i_su^j_mu^k_ny\1$  for $x,y\in A$ and $W_\sigma=u^i_su^j_m, W_\tau=u^j_mu^k_n,$ $W_\mu=u^i_su^k_n$. $r_\tau(r_\mu (r_\sigma(u)))=r_\tau(r_\mu(xu^j_mu^i_su^k_ny\1+xf_\sigma u^k_ny\1))=r_\tau(xu^j_mu^k_nu^i_sy\1+xu^j_mf_\mu y\1+xf_\sigma u^k_ny\1)=xu^k_nu^j_mu^i_sy\1+xf_\tau u^i_sy\1+xu^j_mf_\mu y\1+xf_\sigma u^k_ny\1$ and $r_\sigma(r_\mu (r_\tau(u)))=r_\sigma(r_\mu(xu^i_su^k_nu^j_my\1+xu^i_sf_\tau y\1))=r_\sigma(xu^k_nu^i_su^j_my\1+xf_\mu u^j_my\1+xu^i_sf_\tau y\1)=xu^k_nu^j_mu^i_sy\1+xu^i_sf_\tau y\1+xf_\mu u^j_my\1+xu^k_nf_\sigma y\1$.  $r_\tau(r_\mu (r_\sigma(u)))-r_\sigma(r_\mu (r_\tau(u)))=x([[u^i_s,u^j_m],u^k_n]+[[u^j_m,u^k_n],u^i_s]\\+[[u^k_n,u^i_s],u^j_m])y\1$.

Case 2. $W_\sigma$ and $W_\tau$ are disjoint in $u$. Then $u=xW_\sigma yW_\tau z\1$ where $x,y,z\in A$. It is clear that $r_\sigma(r_\tau(u))=r_\tau (r_\sigma(u))$.  Suppose $r_1r_\tau(u)$, $r_2r_\sigma r_\tau(u)$, $r_3r_\sigma (u)$ and $r_4r_\tau r_\sigma (u)$ are irreducible, where $r_1,r_2,r_3,r_4$ are compositions of some reductions. $r_1r_\tau(u)-r_3r_\sigma (u)=(r_1r_\tau(u)-r_2r_\sigma r_\tau(u))+(r_2r_\sigma r_\tau(u)-r_4r_\tau r_\sigma (u))+(r_4r_\tau r_\sigma (u)-r_3r_\sigma (u))$. We know that $r_\sigma (u), r_\tau(u)$ and $r_\sigma r_\tau(u)=r_\tau r_\sigma(u)$ have smaller formal lengths than that of $u$, hence we can use induction on formal lengths to get that $r_1r_\tau(u)-r_3r_\sigma (u)$ can be written as linear combinations of element in (\ref{reduction1}).

In \cite{Zha2014}, we have the following lemma.
\begin{lemma}For a vertex algebra $W$, $u,v,w\in W$ and $l,m,n\in \Z$ $$[[u_l,v_m],w_n]+[[v_m,w_n],u_l]+[[w_n,u_l],v_m]$$ $$=\sum_{i,j\geq 0}{l\choose j}{m\choose
i}(v_iu_jw-u_jv_iw+[u_j,v_i]w)_{l+m+n-i-j}.$$
\end{lemma}
Use above lemma, (\ref{reduction1}) can be written as
\begin{equation}\label{reduction2}x(u^i_su^j_mu^k-u^j_mu^i_su^k+[u^i_s,u^j_m]u^k)_{t}y\1,
\end{equation}
where $t\in\Z$. Apply skew symmetry (\ref{skewsym}) or commutator formula (\ref{bracket}) to (\ref{reduction2}),  we can get the following lemma,
\begin{lemma}\label{reduction3} The linear relations of $U'_{irr}$ can be written as the images under reduction of the following elements:
\begin{equation}\label{reduction3}x(u^i_su^j_mu^k-u^j_mu^i_su^k+[u^i_s,u^j_m]u^k)_{t}\1,
\end{equation} where $x\in A$, $t<-1$.
\end{lemma}

Now we compare (\ref{reduction3}) with (\ref{idealele}), we find that the kernel of the map from $U'_{irr}$ to $U$ and an ideal of an non-degenerate case are similar. So we give the following definition.
\begin{definition}For a $C_1$-cofinite vertex algebra $V$, $u^1,\cdots, u^l$ are its $C_1$-generators, and $R(i,j,m)$ are their $C_1$-relations, if the $C_1$-relations are degenerate, then we call the following elements
\begin{equation}u^i_su^j_mu^k-u^j_mu^i_su^k+[u^i_s,u^j_m]u^k
\end{equation}
as \emph{$C_1$-singular elements}, where $1\leq i, j,k\leq l$, $s,m\geq 0$.
\end{definition}
It is clear that there are only finite many $C_1$-singular elements and they quite like singular elements in non-degenerate case.

Now we study $A(U)$, and know that $A(U)$ is generated by $o(u^1),\cdots, o(u^l)$, so we need find all the relations of $o(u^1),\cdots, o(u^l).$ We have the obvious relations coming from $C_1$-relations $$[o(u^i),o(u^j)]=\sum_{k\geq 0}{\wt (u^i)-1\choose k}o(R(i,j,k)).$$
But there are some other relations caused by the failure of Jacobi identity.
First we construct an admissible $U$-module like non-degenerate case. We use the same notation as last section(Proposition \ref{au}). Similar to the results about $U$, we get linear relations in $\tilde{N}_{irr}$.
\begin{lemma}\label{lemmatildenirr}The linear relations of $\tilde{N}_{irr}$ can be written as the images under reduction of the following elements:
\begin{equation}\label{reduction4}x(u^i_su^j_mu^k-u^j_mu^i_su^k+[u^i_s,u^j_m]u^k)_{t}\1,
\end{equation} where $x\in A$.
\end{lemma}
Since reduction keeps the weight, we have the similar result for $\tilde{N}_{0irr}$.
\begin{lemma}\label{lemmatildenirr1}The linear relations of $\tilde{N}_{0irr}$ can be written as the images under reduction of the following elements:
\begin{equation}\label{reduction5}x(u^i_su^j_mu^k-u^j_mu^i_su^k+[u^i_s,u^j_m]u^k)_{t}\1,
\end{equation} where $x\in A$ and the weight is $0$.
\end{lemma}

\begin{lemma} $N_0\simeq A(U)$ as vector spaces. \end{lemma}
\begin{proof}From our construction of $N$, we can see that $N$ is the biggest admissible $U$-module with one generator. So $N_0$ is the biggest $A(U)$-module with one generator. Hence $N_0\simeq A(U)$ as vector spaces.
\end{proof}
Furthemore we get the following proposition.

\begin{proposition}$A(U)$ is generated by $o(u^1),\cdots, o(u^l)$ with the relations
$$[o(u^i),o(u^j)]=\sum_{k\geq 0}{\wt (u^i)-1\choose k}o(R(i,j,k))$$ and
$$o(u^{i_1}_{n_1}\cdots u^{i_h}_{n_h}(u^i_mu^j_nu^k-u^j_nu^i_mu^k-[u^i_m,u^j_n]u^k)\1),$$ where $0\leq n_1\leq \cdots\leq n_h$.
\end{proposition}
\begin{proof} The first half part is obvious. From above lemma, we know that $N_0$ is just the free $A(U)$-module with generator $\1$. So if we reduce elements in (\ref{reduction5}) by the reduction system, we will get the result.
\end{proof}
By above proposition, we can see that for the degenerate universal vertex algebra $U$, the calculation of $A(U)$ is similar to the calculation of $A(V)$, where $V$ is a quotient of a non-degenerate universal vertex algebra, if we treat the $C_1$-sigular elements like the usual singular elements. We use an example to explain the process.
\begin{example}We consider a vertex algebra which has three $C_1$-generators $\alpha, e^\alpha,e^{-\alpha}$. Let $\wt \alpha=1$ and $\wt e^\alpha =\wt e^{-\alpha}=2$. Their $C_1$-relations are given as follows:
\begin{align*}
\alpha_1\alpha&=4\1,\\
\alpha_0e^{\alpha}&=4e^{\alpha},\\
\alpha_0e^{-\alpha}&=-4e^{-\alpha},\\
e^{\alpha}_3e^{-\alpha}&=\1,\\
e^{\alpha}_2e^{-\alpha}&=\alpha,\\
e^{\alpha}_1e^{-\alpha}&=\frac{\alpha_{-2}\1}{2}+\frac{\alpha_{-1}\alpha}{2},\\
e^{\alpha}_0e^{-\alpha}&=\frac{\alpha_{-3}\1}{3}+\frac{\alpha_{-2}\alpha}{2}+\frac{\alpha_{-1}\alpha_{-1}\alpha}{6},
\end{align*}
other relations which we do not list are all $0$. Hence by Theorem \ref{abconva}, we get a universal vertex algebra $V$. By direct calculations, we find all $C_1$-singular vectors as following:
\begin{align*}
e^{\alpha}_1e^{\alpha}_0e^{-\alpha}-e^{\alpha}_0e^{\alpha}_1e^{-\alpha}-[e^{\alpha}_1,e^{\alpha}_0]e^{-\alpha}&=10(\alpha_{-1}e^{\alpha}-e^{\alpha}_{-2}\1),\\
e^{-\alpha}_1e^{-\alpha}_0e^{\alpha}-e^{-\alpha}_0e^{-\alpha}_1e^{\alpha}-[e^{-\alpha}_1,e^{-\alpha}_0]e^{\alpha}&=-10(e^{-\alpha}_{-2}\1+\alpha_{-1}e^{-\alpha}).
\end{align*}
Hence $V$ is degenerate. Let $J=e^{\alpha}_{-2}\1-\alpha_{-1}e^{\alpha}$ and $L=e^{-\alpha}_{-2}\1+\alpha_{-1}e^{-\alpha}.$ By direct calculations, we have the following:
\begin{align}
\alpha _0 J&=4J,\notag \\
\alpha_n J&=0  \,\, {\textrm{for}} \,\,n\geq 1, \notag\\
e^{\alpha}_nJ&=0 \,\, {\textrm{for}} \,\,n\geq 1,\notag\\
e^{-\alpha}_nJ&=0  \,\, {\textrm{for}} \,\,n\geq 1,\notag\\
e^{\alpha}_0J&=4e^{\alpha}_{-1}e^{\alpha},\label{e1}\\
e^{-\alpha}_0J&=-4e^{-\alpha}_{-1}e^{\alpha}-\alpha_{-4}\1+\frac{\alpha_{-2}\alpha_{-2}\1}{2}+\frac{4\alpha_{-3}\alpha}{3}-\alpha_{-2}\alpha_{-1}\alpha+\frac{\alpha_{-1}\alpha_{-1}\alpha_{-1}\alpha}{6},\label{e2}\\
e^{\alpha}_0e^{\alpha}_0J&=0,\notag\\
e^{-\alpha}_0e^{\alpha}_0J&=\frac{20J_{-3}\1}{3}-\frac{5\alpha_{-1}J_{-2}\1}{3}+\frac{14\alpha_{-2}J}{3}+\frac{\alpha_{-1}\alpha_{-1}J}{3},\notag\\
e^{-\alpha}_0e^{-\alpha}_0J&=\frac{20L_{-3}\1}{3}+\frac{50\alpha_{-1}L_{-2}\1}{3}-\frac{80\alpha_{-2}L}{3}+\frac{10\alpha_{-1}\alpha_{-1}L}{3},\notag\\
\alpha _0 L&=-4L,\notag\\
\alpha_n L&=0  \,\, {\textrm{for}} \,\,n\geq 1,\notag\\
e^{\alpha}_nL&=0 \,\, {\textrm{for}} \,\,n\geq 1,\notag\\
e^{-\alpha}_nL&=0 \,\, {\textrm{for}} \,\,n\geq 1,\notag\\
e^{-\alpha}_0L&=4e^{-\alpha}_{-1}e^{-\alpha},\label{e3}\\
e^{\alpha}_0L&=-4e^{\alpha}_{-1}e^{-\alpha}+\alpha_{-4}\1+\frac{\alpha_{-2}\alpha_{-2}\1}{2}+\frac{4\alpha_{-3}\alpha}{3}+\alpha_{-2}\alpha_{-1}\alpha+\frac{\alpha_{-1}\alpha_{-1}\alpha_{-1}\alpha}{6},\label{e4}\\
e^{-\alpha}_0e^{-\alpha}_0L&=0,\notag\\
e^{\alpha}_0e^{-\alpha}_0L&=\frac{20L_{-3}\1}{3}+\frac{5\alpha_{-1}L_{-2}\1}{3}-\frac{14\alpha_{-2}L}{3}+\frac{\alpha_{-1}\alpha_{-1}L}{3},\notag\\
e^{\alpha}_0e^{\alpha}_0L&=\frac{20J_{-3}\1}{3}-\frac{50\alpha_{-1}J_{-2}\1}{3}+\frac{80\alpha_{-2}J}{3}+\frac{10\alpha_{-1}\alpha_{-1}J}{3}.\notag
\end{align}
By Proposition \ref{avgenerator}, $A(V)$ is generated by $o(\alpha),o(e^{\alpha}),o(e^{-\alpha})$ which satisfy the relations caused by commutators:
\begin{align}
[o(\alpha),o(e^{\alpha})]&=4o(e^{\alpha}),\notag\\
[o(\alpha),o(e^{-\alpha})]&=-4o(e^{-\alpha}),\notag\\
[o(e^{\alpha}),o(e^{-\alpha})]&=\frac{o(\alpha)o(\alpha)o(\alpha)-o(\alpha)}{6},\notag
\end{align}
and relations caused by $C_1$-singular vectors (only$(\ref{e1},\ref{e2},\ref{e3},\ref{e4})$ are effective):
$$o(e^{-\alpha})o(\alpha)-2o(e^{-\alpha}),$$
$$o(e^{\alpha})o(\alpha)+2o(e^{\alpha}),$$
$$o(e^{\alpha})o(e^{\alpha}),$$
$$-4o(e^{\alpha})o(e^{-\alpha})-\frac{o(\alpha)}{3}-\frac{o(\alpha)^2}{6}+\frac{o(\alpha)^3}{3}+\frac{o(\alpha)^4}{6},$$
$$-4o(e^{-\alpha})o(e^{\alpha})+\frac{o(\alpha)}{3}-\frac{o(\alpha)^2}{6}-\frac{o(\alpha)^3}{3}+\frac{o(\alpha)^4}{6},$$
$$o(e^{-\alpha})o(e^{-\alpha}).$$
$A(V)$ is an associative algebra with basis $\{\1,o(\alpha),o(e^{\alpha}),o(e^{-\alpha}),o(e^{\alpha})o(e^{-\alpha}),o(\alpha)^2,o(\alpha)^3\}$ and is semi-simple. The generators can be represented as matrices as following,
\begin{flalign*}
o(\alpha)&=\begin{pmatrix}0& 0 & 0 & 0  & 0\\0 & 2 & 0 & 0  & 0  \\0 & 0 & -2 & 0  & 0\\ 0 & 0 & 0 & 1  & 0\\                         0 & 0 & 0 & 0  & -1\end{pmatrix},\\
o(e^\alpha)&=\begin{pmatrix}0& 0 & 0 & 0 & 0\\0&0&1&0&0\\0&0&0&0&0\\0&0&0&0&0\\0&0&0&0&0\end{pmatrix},\\
o(e^{-\alpha})&=\begin{pmatrix}0& 0 & 0 & 0 & 0\\0&0&0&0&0\\0 & 1 & 0 & 0 & 0  \\
                         0 & 0 & 0 & 0 & 0  \\
                         0 & 0 & 0 & 0 & 0
                       \end{pmatrix}.
\end{flalign*}

\end{example}
\begin{remark}This example is in fact the vertex operator algebra associated to one dimensional lattice $\Z\alpha$, such that $<\alpha,\alpha>=4$, whose structures are studied in \cite{Frelepmeu1988} and \cite{Lepli2004}, and its representations are studied in \cite{don1993}, where the author find their modules without using Zhu's algebra.
\end{remark}

Now we consider a quotient $V$ of $U$ by an ideal $I$, $V=A/I$. Usually $I$ is generated by some singular elements $a^1,\cdots, a^k \in U$. The question becomes to find $A(I)$, so like the non-degenerate case, we only need to find $o(u^{i_1}_{n_1}\cdots u^{i_r}_{n_r}a^j)$, where $1\leq i_1,\cdots,i_r\leq l$, $0\leq n_1\leq \cdots \leq n_r$ and $1\leq j\leq k$. In fact we can put all the $C_1$-singular and singular vectors together, and treat them equally and use the same procedure as last two examples to find the relations of $o(u^i)$.

\textbf{Acknowledgement:} The authors thank professor Haisheng Li for introducing us Diamond Lemma. Lu Ding and Wei Zhang thank Morningside Center of Mathematics, Chinese Academy of Sciences for providing excellent research environment. Wei Zhang is supported  by China NSF grant(11101030) and Beijing Higher Education Young Elite Teacher Project. Wei Jiang is Supported by China NSF grant (11271056), the NSF grant (10KJD110001) and Qing Lan Project from Jiangsu Provincial Education Department and the ZJNSF Grant (LQ12A01005).
\par

\bibliographystyle{amsplain}
\bibliography{references}

\end{document}